\documentclass{article}

\usepackage{amsthm, amssymb,amsmath}
\usepackage{mathscinet}
\usepackage[utf8]{inputenc}
\usepackage{hyperref}
\usepackage{graphicx}
\usepackage{subcaption}
\usepackage{dsfont}
\usepackage{float}

\newtheorem{theorem}{Theorem}

\newtheorem{proposition}{Proposition}
\newtheorem{lemma}{Lemma}

\newtheorem*{definition*}{Definition}
\newtheorem*{theorem*}{Theorem}

\newtheorem{claim}{Claim}

\newcommand{\E}[1]{\mathbb{E}\left(#1\right)}

\newcommand{\Pof}[1]{\mathbb{P}\left(#1\right)}

\renewcommand{\P}{\mathbb{P}}

\newcommand{\R}{\mathbb{R}}
\newcommand{\Z}{\mathbb{Z}}

\newcommand{\F}{\mathcal{F}}

\newcommand{\one}[1]{\mathds{1}_{#1}}

\DeclareMathOperator{\GL}{GL}
\DeclareMathOperator{\Flags}{Flags}

\newcommand{\dist}{\operatorname{dist}}

\newcommand{\flags}{\Flags}
\renewcommand{\det}{\text{det}}

\newcommand{\eqd}{\stackrel{(d)}{=}}
\newcommand{\lowerdim}{\underline{\dim}}
\newcommand{\upperdim}{\overline{\dim}}

\begin{document}

\title{Entropy and dimension of disintegrations of stationary measures}

\author{Pablo Lessa\thanks{IMERL, Facultad de Ingeniería, Julio Herrera y Reissig 565,  11300 Montevideo, Uruguay.  Email: plessa@fing.edu.uy}}

\maketitle

\begin{abstract}
We extend a result of Ledrappier, Hochman, and Solomyak on exact dimensionality of stationary measures for \(\text{SL}_2(\R)\) to disintegrations of stationary measures for \(\GL(\R^d)\) onto the one dimensional foliations of the space of flags obtained by forgetting a single subspace.

The dimensions of these conditional measures are expressed in terms of the gap between consecutive Lyapunov exponents, and a certain entropy associated to the group action on the one dimensional foliation they are defined on.   It is shown that the entropies thus defined are also related to simplicity of the Lyapunov spectrum for the given measure on \(\GL(\R^d)\).
\end{abstract}

\section{Introduction}

It was shown by Ledrappier \cite{ledrappier}, Hochman and Solomyak \cite{hochman-solomyak2017}, that if \(\nu\) is a probability on the projective space of \(\R^2\) which is stationary with respect to a probability \(\mu\) on \(\text{SL}_2(\R)\) with finite Lyapunov exponents, then \(\nu\) is exact dimensional and its dimension is \(\frac{\kappa}{2\chi}\) where \(\kappa\) is the Furstenberg entropy and \(\chi\) is the largest Lyapunov exponent (hence \(2\chi\) is the gap between the two Lyapunov exponents).

Suppose now that \(\mu\) is a probability on \(\text{SL}_3(\R)\) and \(\nu\) is a \(\mu\)-stationary probability on the space of flags in \(\R^3\) (i.e. pairs \((L,P)\) where \(L \subset P\), \(L\) is a one dimensional subspace, and \(P\) is a two dimensional subspace), which is a three-dimensional manifold.

We consider here the two foliations of the space of flags obtained by partitioning into sets of flags sharing the same one dimensional subspace on the one hand, and flags sharing the same two dimensional subspace on the other.   These are foliations by circles, and furthermore the action of any invertible linear self mapping of \(\R^3\) preserves both foliations.

In this context we show that the conditional measures obtained by disintegrating \(\nu\) with respect to these two foliations, are exact dimensional.  Furthermore we express the dimension of these disintegrations in terms of the gap between consecutive Lyapunov exponents as well as two entropies \(\kappa_1,\kappa_2\).  Before establishing the dimension formula we show that the entropies \(\kappa_i\) bound the gaps between exponents from below and therefore, in principle, yield a criteria for simplicity of the Lyapunov spectrum.

We prove our results in a slightly more general context, that of actions of \(\GL(\R^d)\) on the space of complete flags in \(\R^d\).  In this context there are \(d-1\) associated one dimensional foliations which correspond to ``forgetting'' the \(i\)-dimensional subspace of all flags for some \(i \in \lbrace 1,\ldots, d-1\rbrace\).

\subsection{Preliminaries}

 Let \(\sigma_1(A) \ge \sigma_2(A) \ge \cdots \ge \sigma_d(A) > 0\) denote the singular values of an element \(A \in \GL(\R^d)\) with respect to the standard inner product.  
 
 We denote by \(\Flags(\R^d)\) the space of complete flags in \(\R^d\), an element \(F \in \Flags(\R^d)\) is of the form \(F = (S_0,S_1,\ldots,S_d)\) where \(S_i\) is an \(i\)-dimensional subspace of \(\R^d\) for each \(i = 0,\ldots,d\) and \(S_i \subset S_{i+1}\) for \(i = 0,\ldots,d-1\). 

 Let \(\Flags_i(\R^d)\) denote the space of flags missing their \(i\)-dimensional subspace.  For a given complete flag \(F = (S_0,\ldots,S_d)\) we denote by \(F_i\) its projection to \(\Flags_i(\R^d)\) (i.e. the sequence obtained by removing \(S_i\) from \(F\)).
 
 We use the notation \(X \eqd Y\) for equality in distribution between random elements \(X\) and \(Y\). And \(\nu_1 \ll \nu_2\) to mean that the probability \(\nu_1\) is absolutely continuous with respect to \(\nu_2\).
 
 If \(X\) and \(Y\) are random elements taking values in complete separable metric spaces (a version of) the conditional distribution of \(X\) given \(Y\) is a \(\sigma(Y)\)-measurable random probability \(\nu_Y\) on the range of \(X\) such that
 \[\int f(x)d\nu_Y(x) = \E{f(X)|Y}\]
 for all continuous bounded real functions (here the right-hand side is the conditional expectation of \(f(X)\) with respect to the \(\sigma\)-algebra generated by \(Y\)).   Such a conditional distribution is well defined up to sets of zero measure but we will abuse notation slightly referring to `the conditional distribution'.

 It is always the case that there exists a Borel mapping \(y \mapsto \nu(y)\) from the range of \(Y\) to the space of probabilities on the range of \(X\) such that \(\nu(Y)\) is a version of the conditional distribution of \(X\) given \(Y\).  Fixing such a mapping one may speak of \(\nu_y\) for \(y\) non-random in the range of \(Y\).

 The lower local dimension of a probability measure \(\nu\) on a metric space at a point \(x\) is defined by
 \[\lowerdim_x(\nu) = \liminf\limits_{r \to 0}\frac{\log\left(\nu(B_r(x)\right)}{\log(r)},\]
 while the upper local dimension is defined by
 \[\upperdim_x(\nu) = \limsup\limits_{r \to 0}\frac{\log\left(\nu(B_r(x)\right)}{\log(r)},\]
 where \(B_r(x)\) is the ball of radius \(r\) centered at \(x\).
 
 If the lower and upper dimensions of \(\nu\) are equal to the same constant \(\nu\)-almost everywhere then we say that \(\nu\) is exact dimensional and define its global dimension \(\dim(\nu)\) as the given constant.
 
\subsection{Statement of main results} 

 Suppose that \(A\) is a random element of \(\GL(\R^d)\) with distribution \(\mu\) such that
 \[\E{\left|\log\left(\sigma_i(A)\right)\right|} < +\infty\text{ for }i= 1,\ldots,d,\]
 and let \(F = (S_0,\ldots,S_d)\) be a random element of \(\Flags(\R^d)\) with distribution \(\nu\) which is independent from \(A\) and such that
 \[F \eqd AF.\]
 
 The existence of such a pair \((A,F)\) is equivalent to the fact that \(\nu\) is a \(\mu\)-stationary probability, as first defined in \cite{furstenberg1963}.

 The Lyapunov exponents \(\chi_1,\ldots,\chi_d\) of \(\mu\) relative to \(\nu\) are defined by the equations
 \[\chi_1+\cdots+\chi_i = \E{\log\left(\left|\det_{S_i}(A)\right|\right)}\text{ for }i=1,\ldots,d,\]
 where \(|\det_S(A)|\) is the Jacobian of the restriction of \(A\) to the subspace \(S\) (where the volume measure induced by standard inner product is used on \(S\) and its image).  In the degenerate case where \(S = \lbrace 0\rbrace\) one has \(|\det_S(A)| = 1\), and if \(S\) is one dimensional one has \(|\det_S(A)| = \|A_{|S}\|\).  
 
 The Lyapunov exponents given by the multiplicative ergodic theorem of \cite{oseledets} for a product of i.i.d. random matrices of distribution \(\mu\) are obtained by maximizing the sums \(\chi_1 + \cdots + \chi_i\) over all stationary probabilities \(\nu\) as shown in \cite{furstenberg-kifer}.

 Fix \(i \in \lbrace 1,\ldots, d-1\rbrace\), let \(\nu_i\) be the projection of \(\nu\) to \(\Flags_i(\R^d)\), and let \(\nu_{F_i}\) be the conditional distribution of \(F\) given \(F_i\).

\begin{theorem}[Inequality between entropy and gap between exponents]\label{gaptheorem}
 If \(\nu\) is the unique stationary probability on \(\Flags(\R^d)\) which projects to \(\nu_i\) then \(A\nu_{F_i} \ll \nu_{AF_i}\) almost surely,
 \[0 \le \kappa_i = \E{\log\left(\frac{dA\nu_{F_i}}{d\nu_{AF_i}}\left(AF\right)\right)} \le \chi_i - \chi_{i+1},\]
 and \(\kappa_i = 0\) if and only if \(A\nu_{F_i} = \nu_{AF_i}\) almost surely.
\end{theorem}

\begin{theorem}[Dimension of conditional measures]\label{dimensiontheorem}
 If \(\nu\) is ergodic, is the unique stationary probability on \(\Flags(\R^d)\) which projects to \(\nu_i\), and \(\kappa_i > 0\), then almost surely \(\nu_{F_i}\) is exact dimensional and
 \[\dim(\nu_{F_i})= \frac{\kappa_i}{\chi_i - \chi_{i+1}}.\]
\end{theorem}

In the case \(d = 2\) both theorems above are known. A proof of Theorem \ref{gaptheorem} in this case was first given in \cite{ledrappier}.  In the same work the formula for dimension in Theorem \ref{dimensiontheorem} is shown to hold for a slightly different notion of dimension.   The exact dimensionality of stationary measures when \(d = 2\) was first proved in \cite{hochman-solomyak2017} and this implies the formula above for the same notion of dimension we use here.

Theorem \ref{gaptheorem} implies that the Lyapunov spectrum is simple (i.e. all exponents are different) if there does not exist a family of conditional probabilities \(F_i \mapsto \nu_{F_i}\) satisfying \(A\nu_{F_i} = \nu_{AF_i}\) for \(\mu\) almost every \(A\). This suggests a connection to criteria for simplicity dating back to \cite{goldsheid-margulis} and \cite{guivarch-raugi} though we do not explore this issue further here.

\subsection{Acknowledgment}

I am grateful to François Ledrappier for many helpful discussions.  I would also like to thank an anonymous referee for pointing out an error in a previous version of the proof of theorem 1, and for helping improve the general quality of the article.

\part{Entropy, Mutual information, and Lyapunov exponent gaps}

\section{Entropy and mutual information}

We will define below \(I(A,AF|AF_i)\) the conditional mutual information between \(A\) and \(AF\) given \(AF_i\).  This is a non-negative \(\sigma(AF_i)\)-measurable random variable which may take the value \(+\infty\).

The purpose of this section is to prove that:
\begin{lemma}[Entropy and mutual information]\label{entropyandmutualinformationlemma}
If \(I(A,AF|AF_i) < +\infty\) almost surely then \(A\nu_{F_i} \ll \nu_{AF_i}\) almost surely and \(\kappa_i = \E{I(A,AF|AF_i)}\).

Conversely, if \(A\nu_{F_i} \ll \nu_{AF_i}\) almost surely then \(\kappa_i = \E{I(A,AF|AF_i)}\) whether \(\kappa_i\) is finite or not.
\end{lemma}

This result reduces the problem of showing that \(A\nu_{F_i} \ll \nu_{AF_i}\) almost surely and that \(0 \le \kappa_i < +\infty\) to that of bounding the conditional mutual information between \(A\) and \(AF\) given \(AF_i\).

A general reference covering mutual information including Dobrushin's theorem and the Gelfand-Yaglom-Perez theorem is \cite{pinsker}.

\subsection{Conditional mutual information}

\subsubsection{Mutual information}

Let \(X\) and \(Y\) be random elements of two Polish spaces \(\mathcal{X}\) and \(\mathcal{Y}\), and denote \(\mu_X,\mu_Y,\mu_{(X,Y)}\) the distribution of \(X\), \(Y\), and \((X,Y)\) respectively.

The mutual information between \(X\) and \(Y\) is defined by
\[I(X,Y) = \sup \sum\limits_{A \in P}\log\left(\frac{\mu_{(X,Y)}(A)}{(\mu_X \times \mu_Y)(A)}\right)\mu_{(X,Y)}(A)\]
where the supremum is over all finite partitions \(P\) of \(\mathcal{X}\times \mathcal{Y}\) into Borel sets.

Directly from the definition one sees that \(I(X,Y) = I(Y,X)\).   

By Jensen's inequality \(0 \le I(X,Y) \le +\infty\) with equality to \(0\) if and only if \(X\) and \(Y\) are independent.   If \(X\) takes countably many values and has finite entropy \(H(X)\) in the sense of \cite{shannon}  one has \(I(X,Y) \le H(X)\).

It was shown in \cite{dobrushin} that \(I(X,Y)\) is the supremum over any sequence of partitions which generate the Borel \(\sigma\)-algebra in \(\mathcal{X}\times \mathcal{Y}\) (see also \cite[Lemma 7.3]{gray}).  This has the following important corollary:
\begin{proposition}[Semi-continuity of mutual information]\label{semicontinuityproposition}
 If \(\lim\limits_{n \to +\infty}(X_n,Y_n) = (X,Y)\) in the sense of distributions then \(I(X,Y) \le \liminf\limits_{n \to +\infty} I(X_n,Y_n)\).
\end{proposition}

It was shown in \cite{gelfand-yaglom} and \cite{perez} that if \(I(X,Y) < +\infty\) then \(\mu_{(X,Y)} \ll \mu_X \times \mu_Y\) and 
\[I(X,Y) = \E{\log\left(\frac{d\mu_{(X,Y)}}{d(\mu_X\times \mu_Y)}(X,Y)\right)}.\]

Conversely, if \(\mu_{(X,Y)} \ll \mu_X \times \mu_Y\) then
\[I(X,Y) = \E{\log\left(\frac{d\mu_{(X,Y)}}{d(\mu_X\times \mu_Y)}(X,Y)\right)},\]
whether the right hand side is finite or not.

These results are usually called the Gelfand-Yaglom-Perez Theorem.

In our context, when \(d = 2\), this yields the following result:
\begin{proposition}\label{twodimensionalinformation}
If \(d = 2\) and \(I(A,AF) < \infty\) then \(A\nu \ll \nu\) almost surely and \(0 \le \kappa = \E{\log\left(\frac{dA\nu}{d\nu}(AF)\right)} = I(A,AF) < +\infty\).   

Conversely, if \(A\nu \ll \nu\) almost surely then \(\kappa = I(A,AF)\) whether \(I(A,AF)\) is finite or not.
\end{proposition}
\begin{proof}
 The marginal distributions of \((A,AF)\) are \(\mu\) and \(\nu\) respectively.  However the conditional distribution of \(AF\) given \(A\) is \(A\nu\).  
 
 Therefore letting \(m\) be the joint distribution of \((A,AF)\) one has
 \[\int f(a,x) dm(a,x) = \int \int f(a,x) da\nu(x) d\mu(a),\]
 for all measurable functions \(f\).
 
 If \(A\nu \ll \nu\) almost surely then
 \begin{align*}
\int f(a,x) dm(a,x) &= \int \int f(a,x) \frac{da\nu}{d\nu}(x) d\nu(x) d\mu(a)
\\ & = \int f(a,x) \frac{da\nu}{d\nu}(x) d(\mu\times \nu)(a,x),  
 \end{align*}
 so that \(\frac{dm}{d(\mu \times \nu)}(a,x) = \frac{da\nu}{d\nu}(x)\) at \((\mu\times \nu)\)-almost every point \((a,x)\).
 
 On the other hand if \(m \ll (\mu \times \nu)\) then setting \(g(a,x) = \frac{dm}{d(\mu\times \nu)}(a,x)\) one has
 \[\int f(a,x) dm(a,x) = \int \int f(a,x) da\nu(x) d\mu(a) = \int \int f(a,x) g(a,x) d\nu(x) d\mu(a),\]
 for all measurable funtions \(f\).   
 
 Letting \(f(a,x) = \one{A}(a)h(x)\) where \(\one{A}\) is the indicator of an arbitrary subset \(A\) of \(\GL(\R^d)\), and \(h\) is continuous on the compact space \(\flags(\R^2)\),  one obtains that
 \[\int h(x) da\nu(x) = \int h(x) g(a,x) d\nu(x),\]
 for \(\mu\)-almost every \(a\).   Intersecting the \(\mu\)-full measure sets where this holds over a countable dense set of functions \(h\), one obtains a full measure set for \(\mu\) where 
 \(\frac{da\nu}{d\nu}(x) = g(a,x)\).
 
 Hence, the distribution of \((A,AF)\) is absolutely continuous with respect to \(\mu \times \nu\) if and only if \(A\nu \ll \nu\) almost surely and in this case the Radon-Nikodym derivative between the two at \((A,AF)\) is given by \(\frac{dA\nu}{d\nu}(AF)\).
\end{proof}

\subsubsection{Conditional mutual information}

Let \(\F\) be a \(\sigma\)-algebra of measurable sets in the probability space on which the random elements \(X\) and \(Y\) are defined.

The mutual information between \(X\) and \(Y\) conditioned on \(\F\) is the unique up to modifications on null sets random variable  \(I(X,Y|\F)\) obtained as above but using the conditional distribution of \((X,Y)\) conditioned on \(\F\).  In the case \(\F = \sigma(Z_1,Z_2,\ldots,Z_k)\) we use the notation \(I(X,Y|Z_1,Z_2,\ldots,Z_k) = I(X,Y|\F)\).

One still has \(0 \le I(X,Y|\F) = I(Y,X|\F) \le +\infty\) almost surely.   Almost sure equality to zero occurs if and only if \(X\) and \(Y\) are conditionally independent given \(\F\).

In general there is no relation between \(I(X,Y)\) and \(I(X,Y|\F)\) or even \(\E{I(X,Y|F)}\).

To see this suppose for example that \(X,Y\) are i.i.d. taking the values  \(\pm 1\) with probability \(1/2\) and \(Z = XY\), then one has \(I(X,Y) = 0\) while \(I(X,Y|Z) = \log(2)\) almost surely.

On the other hand for any Markov chain \(X_1,X_2,X_3\) one has \(I(X_1,X_3|X_2) = 0\) almost surely, and one may construct examples with \(I(X_1,X_3) > 0\).  For example, setting \(X_1 = Y_1 , X_2 = Y_1+Y_2\) and \(X_3 = Y_1+Y_2+Y_3\) where the \(Y_i\) are i.i.d. with \(\P(Y_i = \pm 1) = 1/2\) suffices.

The following semi-continuity property holds:
\begin{proposition}[Semi-continuity of conditional mutual information]\label{conditionalsemicontinuityproposition}
If the conditional distribution of \((X_n,Y_n)\) given \(\F\) converges almost surely to the conditional distribution of \((X,Y)\) given \(\F\) then \(I(X,Y|\F) \le \liminf\limits_{n \to +\infty} I(X_n,Y_n|\F)\) almost surely. 
\end{proposition}
\begin{proof}
This is a direct consequence of Proposition \ref{semicontinuityproposition}.
\end{proof}

The following monotonicity property follows immediately from the definition of mutual information
\[I(X,Y\vert \F) \le I(X,(Y,Z)\vert \F).\]

A more precise version of monotonicity is the following:
\begin{proposition}[Chain rule for conditional mutual information]\label{cocycleproposition}
If \(X,Y,Z\) are random elements and \(\F\) a \(\sigma\)-algebra of events of the probability space on which they are defined, then 
\[I(X,(Y,Z)\vert \F) = \E{I(X,Y \vert Z,\F)\vert \F} + I(X,Z\vert \F).\] 
\end{proposition}
\begin{proof}
 When \(\F\) is trivial this is \cite[Corollary 7.14]{gray} (notice that what said reference denotes by \(I(X,Y|Z)\) is \(\E{I(X,Y|Z)}\) in our notation).   The general case follows by applying this to the conditional distributions given \(\F\).
\end{proof}

\subsection{Proof of Lemma \ref{entropyandmutualinformationlemma}}

We will calculate the marginal distributions and the joint distribution of \((A,AF)\) conditioned on \(AF_i\) and apply the Gelfand-Yaglom-Perez Theorem as in Proposition \ref{twodimensionalinformation}.

To begin we simply let \(\mu_{AF_i}\) be the conditional distribution of \(A\) given \(AF_i\).

By stationarity of \(\nu\) the conditional distribution of \(AF\) given \(AF_i\) is \(\nu_{AF_i}\).

For the joint distribution notice that the distribution of \(AF\) conditioned on \(\sigma(A,AF_i)\) is the same as conditioned on \(\sigma(A,F_i)\) and therefore it is \(A\nu_{F_i}\).

Hence the joint conditional distribution of \((A,AF)\) given \(AF_i\) satisfies (and is determined by the equation)
\[\E{f(A,AF)|AF_i} = \int \int  f(a,x) da\nu_{F_i}(x) d\mu_{AF_i}(a)\]
for all continuous bounded \(f\).

By the Gelfand-Yaglom-Perez Theorem if \(I(A,AF|AF_i) < +\infty\) almost surely then \(A\nu_{F_i} \ll \nu_{AF_i}\) almost surely and 
\[\E{\log\left(\frac{dA\nu_{F_i}}{d\nu_{AF_i}}(AF)\right)|AF_i} < +\infty\]
almost surely.

And conversely, if \(A\nu_{F_i} \ll \nu_{AF_i}\) almost surely one has
\[I(A,AF|AF_i) = \E{\log\left(\frac{dA\nu_{F_i}}{d\nu_{AF_i}}(AF)\right)|AF_i}.\]

The result now follows by taking expectation.

\section{Proof of Theorem \ref{gaptheorem}}

In this section we will prove Theorem \ref{gaptheorem}.  

The strategy is to approximate \((A,F)\) by pairs with the property that the conditional distributions \(\nu_{F_i}\) are absolutely continuous with respect to the natural geometric measure on their domain of definition.

For the approximating pairs there is a direct relation between the distortion of the conditional measures by a linear mapping \(A\) and its determinants on certain subspaces.  This argument establishes equality between the entropy \(\kappa_i\) and the Lyapunov exponent gap \(\chi_i - \chi_{i+1}\) for the approximating pairs.

The result is then obtained by passing to the limit using the properties of conditional mutual information discussed in the previous section.  At this step equality is lost, and one obtains only an inequality between entropy and the Lyapunov exponent gap.  

An important technical issue is that one must maintain the same conditioning \(\sigma\)-algebra for the approximating pairs and the limit pair \((A,F)\) in order to apply Proposition \ref{conditionalsemicontinuityproposition}.

The idea of approximating a probability \(\mu\) by one whose stationary probability is absolutely continuous with respect to the natural geometric measure is already present in \cite[Theorem 8.6]{furstenberg1963}.

\subsection{Jacobians of linear actions on flags}

We will now briefly, for the duration of this subsection, abandon the context where \(A\) and \(F\) are random satisfying \(AF \eqd F\) in order to discuss a result for a deterministic transformation \(A\) and flag \(F\).

Denote the mapping \(F \mapsto F_i\) which removes from each flag in \(\flags(\R^d)\) its \(i\)-dimensional subspace by \(\pi_i\), and notice that the fibers \(\flags_{F_i}(\R^d) = \pi_i^{-1}(F_i)\) are 1-dimensional.   We consider on each \(\flags_{F_i}(\R^d)\) the the unique probability measure \(\eta_{F_i}\) which is invariant under the action of orthogonal transformations which fix \(F_i\).

Notice that any element \(A \in \GL(\R^d)\) leaves the family of measures \(\eta_{F_i}\) quasi-invariant.   We will need the explicit Jacobian of the action of \(A\) on this family of measures.

\begin{lemma}\label{grassmannjacobian}
 If \(A \in \GL(\R^d)\), \(F = (S_0,S_1,\ldots,S_d) \in \flags(\R^d)\), and \(i \in \lbrace 1,\ldots, d-1\rbrace\), then

 \[\frac{dA\eta_{F_i}}{d\eta_{AF_i}}(AF) = \frac{|\det_{S_i}(A)|^2}{|\det_{S_{i-1}}(A)||\det_{S_{i+1}}(A)|}.\]
\end{lemma}
\begin{proof}
 We begin by proving the case \(d = 2\) (this case is included in the statement of \cite[Lemma 8.8]{furstenberg1963} though the proof is omitted there).
 
 In this case \(F = (S_0,S_1,S_2)\) and the only non-trivial subspace is \(S_1\) which has dimension \(1\) in \(\R^2\).   Therefore, we are looking to calculate the Jacobian of the action of \(A\) on the projective space of lines in \(\R^2\) at the line \(S_1\) with respect to the unique rotationally invariant probability \(\eta\).
  
 For this purpose consider a unit length vector \(v \in S_1\) and an orthogonal vector \(w\) of length \(\delta\).  Let \(R\) be the rectangle \(\lbrace sv+tw: s,t \in [0,1]\rbrace\).
 
 Since we are considering the action of \(A\) on projective space, it is equivalent to consider the transformation \(B = A/|\det_{S_1}(A)| = A/|Av|\) so that \(Bv\) has length one.
 
 Notice that \(BR\) is a paralelogram with a side in \(AS_1\) of length \(1\), and area \(\epsilon\) which is the length of the orthogonal projection of \(BR\) onto the subspace orthogonal to \(AS_1\).   Calculating the determinant of \(B\) one obtains explicitly
 \[\epsilon = |\det(B)|\delta = \frac{|\det(A)|}{|\det_{S_1}(A)|^2}\delta.\]
 
 Taking the limit as \(\epsilon \to 0\) we obtain that the derivative of the action of \(A\) on projective space at the point \(S_1\) is \(\frac{|\det(A)|}{|\det_{S_1}(A)|^2}\) from which it follows that
 \[\frac{dA\eta}{d\eta}(AS_1) = \frac{|\det_{S_1}(A)|^2}{|\det(A)|}\]
 as claimed.
 
 We will now show that the general case may be reduced to the two dimensional case.
 
 For this purpose suppose now that \(d > 2\), \(F = (S_0,\ldots,S_d)\), and \(i \in \lbrace 1,\ldots, d-1\rbrace\).   
 
 Notice that the quotient space \(S_{i+1}/S_{i-1}\) is two dimensional and inherits an inner product from \(\R^d\) which makes it isometric to the orthogonal complement of \(S_{i-1}\) within \(S_{i+1}\).  The same is true for \(AS_{i+1}/AS_{i-1}\).
 
 Therefore, letting \(B: S_{i+1}/S_{i-1} \to AS_{i+1}/AS_{i-1}\) be the linear map induced by \(A\) one has
 \[\frac{dA\eta_{F_i}}{d\eta_{AF_i}}(AF) = \frac{|\det_{S_i}(B)|^2}{|\det(B)|},\]
 where on the right hand side the space \(S_i\) is considered as a one-dimensional subspace of \(S_{i+1}/S_{i-1}\).
 
 The result follows from the observation that \(|\det(B)| = |\det_{S_{i+1}}(A)|/|\det_{S_{i-1}}(A)|\) and \(|\det_{S_i}(B)| = |\det_{S_i}(A)|/|\det_{S_{i-1}}(A)|\). 
\end{proof}

\subsection{Proof of Theorem \ref{gaptheorem}}

We return now to the notation and context of the statement of Theorem \ref{gaptheorem}. In particular \(A\) and \(F = (S_0,\ldots,S_d)\) are independent random elements with distribution \(\mu\) and \(\nu\) respectively and such that \(AF \eqd F\).   Recall that \(\nu_i\) is the projection of \(\nu\) onto \(\flags_i(\R^d)\) and \(\nu_{F_i}\) is the conditional distribution of \(F\) given \(F_i\).

\subsubsection{Representation}

Since the statement of the theorem only depends on the joint distribution of \((A,F)\) we are at liberty to change \((A,F)\) to any other pair with the same distribution.

For this purpose fix a Borel mapping \((u,m) \mapsto \rho(u,m)\) where \(u \in [0,1]\), \(m\) is a Borel probability on \(\GL(\R^d)\), and \(\rho(u,m) \in \GL(\R^d)\), such that if \(U\) is a uniformly distributed random variable on \([0,1]\) then \(\rho(U,m)\) has distribution \(m\).

Assume furthermore for any convergent sequence of probabilities \(m_n \to m\) one has \(\rho(U,m_n) \to \rho(U,m)\) almost surely.   Such a representation \(\rho\) exists by the main result of \cite{blackwell-dubins1983}.

In the same way fix a representation \((u,m) \mapsto \rho_{\flags}(u,m)\) into \(\flags(\R^d)\), and representation \((u,m) \mapsto \rho_{\flags_i}(u,m)\) into \(\flags_i(\R^d)\).

Let \(\nu_i\) be the distribution of the incomplete flag \(F_i\), and \(\nu_{F_i}\) the conditional distribution of \(F\) given \(F_i\).

Setting \((A',F_i',F') = (\rho(u_1,\mu),\rho_{\flags_i}(u_2,\nu_i), \rho_{\flags}(u_3,\nu_{F_i'}))\) where \(u_1,u_2,u_3\) are i.i.d. uniform in \([0,1]\), one has that \((A',F')\) has distribution \(\mu\times \nu\) which is the joint distribution of \((A,F)\).

To simplify notation we assume from now on \((A,F) = (A',F')\).

\subsubsection{Perturbation}

Let \(\lbrace R_t, t \ge 0\rbrace\) be defined so that conditioned on \(AF_i\) it is a Brownian motion starting at the identity on the group of orthogonal transformations which fix \(AF_i\).  To clarify dependence on the other random elements we assume \(\lbrace R_t, t \ge 0\rbrace\) is \(\sigma(AF_i,u_4)\)-measurable where \(u_4\) is uniform on \([0,1]\) and independent from \((u_1,u_2,u_3)\).

Now for each \(t \ge 0\) let \(A_t = R_tA\) and notice that \(A_tF_i = AF_i\) almost surely and \(A_t \to A\) when \(t \to 0\) almost surely.

We denote by  \(C(\flags(\R^d),\R)\) the space of real valued continuous functions on \(\flags(\R^d)\) with the topology of uniform convergence, and consider for each \(t \ge 0\) the operator \(P_t:C(\flags(\R^d), \R) \to C(\flags(\R^d),\R)\) defined by
 \[(P_tf)(x) = \E{f(A_tx)}.\]
 
Notice that \(P_t1 = 1\) and if \(f \ge 0\) then \(P_tf \ge 0\).  Therefore there is an associated action of \(P_t\) on the space of probability measures on \(\flags(\R^d)\) defined by
\[\int f(x) d(P_t)^*m(x) = \int (P_tf)(x) dm(x).\]

\begin{lemma}\label{kakutaniargument}
For each \(t > 0\) there is a \(P_t^*\)-invariant probability measure \(\nu_t\) on \(\flags(\R^d)\) whose projection onto \(\flags_i(\R^d)\) is \(\nu_i\).

Furthermore picking for each \(t > 0\) a measure \(\nu_t\) as above one has \(\lim\limits_{t \to 0}\nu_t = \nu\), and letting \(x_i \mapsto \nu_{t,x_i}\) be the disintegration of \(\nu_t\) with respect to the projection to \(\flags_i(\R^d)\) the following properties hold:
 \begin{enumerate}
  \item Almost surely \(\nu_{t,F_i}\) is absolutely continuous with respect to \(\eta_{F_i}\).
  \item There is a compact subinterval \(I_t \subset (0,+\infty)\) such that \(\frac{d\nu_{t,F_i}}{d\eta_{F_i}}\) takes values in \(I_t\) almost surely.
 \end{enumerate}
\end{lemma}
\begin{proof}
Let \(\pi:\flags(\R^d) \to \flags_i(\R^d)\) be the canonical projection.

Let
\(Q_tf(x) = \int f(rx) d\lambda_{t,x_i}(r),\)
where \(x_i = \pi(x)\), and \(\lambda_{t,x_i}\) is the distribution of the time of \(t\) of Brownian motion starting at the identity on the group of orthogonal transformations fixing \(x_i\).

Notice that 
\[P(Q_tf)(x) = \int (Q_tf)(ax) d\mu(a) = \int \int f(rax) d\lambda_{t,ax_i}(r) d\mu(a) = P_tf(x).\]

Since \(Q_t\) preserves the set of functions of the form \(f(x) = g(\pi(x))\) one obtains that \(\pi_*Q_t^*m = \pi_*m\) for all probabilities \(m\).

In particular \(P_t^* = (Q_t)^*P^*\) preserves the space of probabilities which project onto \(\nu_i\).   By the Markov-Kakutani fixed point theorem, this implies that there is at least one fixed point for \(P_t^*\) in this space.

Because \(\lambda_{t,x_i}\) has a continuous positive density with respect to the invariant measure on group of orthogonal transformations stabilizing \(x_i\) it follows that, for any probability \(m\) on \(\flags_i(\R^d)\) the measure \(Q_t^*m\) satisfies properties 1 and 2 in the statement above.

In particular for any \(P_t^*\)-invariant probability \(\nu_t\) with \(\pi_*\nu_t = \nu_i\) one has \(\nu_t = P_t^*\nu_t = Q_t^*(P^*\nu_t)\), and therefore \(\nu_{t}\) satisfies properties 1 and 2.

Finally, let \(f\) be any continuous function and, supose \(m = \lim\limits_{n \to +\infty}\nu_{t_n}\) where \(\lim\limits_{n \to +\infty}t_n = 0\).   Using the notation \(\lambda(f)\) for the integral of \(f\) with respect to the measure \(\lambda\), we have
\begin{align*}
 |m(f) - m(Pf)| &= \lim\limits_{n \to +\infty}|\nu_{t_n}(f) - \nu_{t_n}(Pf)| 
 \\ &= \lim\limits_{n \to +\infty}|\nu_{t_n}(PQ_tf) - \nu_{t_n}(Pf)| 
 \\ &\le \lim\limits_{n \to +\infty}\|P(Q_{t_n}f - f)\|_{\infty} 
 \\ &\le \lim\limits_{n \to +\infty}\|Q_{t_n} f - f\|_{\infty},
\end{align*}
where we have used that \(|Pf(x)| \le  \int |f(ax)| d\mu(a) \le \|f\|_{\infty}\) so \(P\) decreases the \(L^\infty\) norm.

Notice that \(\lambda_{t,x_i}\) converges to the point mass at the identity when \(t \to 0\).  The convergence is uniform in the sense that given \(r > 0\) and letting \(B_r\) be the ball of radius \(r\) centered at the identity in the full orthogonal group, for each \(\epsilon > 0\) there exists \(T > 0\) such that \(\lambda_{t,x_i}(B_r) > 1-\epsilon\) for all \(t < T\) and all \(x_i\).
It follows that
\[\lim\limits_{t \to 0}Q_tf(x) = \lim\limits_{t \to 0}\int f(rx) d\lambda_{t,\pi(x)}(r) = f(x),\]
for all \(x\) and the convergence is uniform.

Since \(\|Q_{t_n} f -f\|_{\infty}\) goes to zero we conclude that \(m(f) = m(Pf)\).   Since this holds for all \(f\) one has that \(P^*m = m\).   By hypothesis \(\nu\) is the unique measure with this property with projection \(\nu_i\), therefore \(m = \nu\).

We have shown that \(\nu\) is the only limit point of \(\nu_t\) when \(t \to 0\).  The space of probabilities on \(\flags(\R^d)\) is compact and metrizable, and therefore this implies \(\lim\limits_{t \to 0}\nu_t = \nu\) as claimed.
\end{proof}

\subsubsection{Conclusion of the proof}

We will fix from now on a sequence \(t_n\) given by the following claim (c.f. \cite[section 6.1.6]{distancestationary}):
\begin{claim}
 There exists a sequence of positive numbers with \(\lim\limits_{n \to +\infty}t_n = 0\) such that, letting \(\nu_{t_n}\) and \(\nu_{t_n,x_i}\) be given by lemma \ref{kakutaniargument} one has
\[\lim\limits_{n \to +\infty}\frac{1}{n}\sum\limits_{j = 1}^n \nu_{t_j,F_i} = \nu_{F_i}\]
 almost surely.
\end{claim}
\begin{proof}
To begin fix any sequence of positive numbers with \(\lim\limits_{m \to +\infty}s_m = 0\).

Let \(\lbrace f_k: k = 1,2,\ldots\rbrace\) be a dense sequence of continuous functions on \(\flags(\R^d)\). 

Notice that \(x_i \mapsto \nu_{s_m,x_i}(f_j), m = 1,2,\ldots\) is a bounded sequence in \(L^1(\flags_i(\R^d),\nu_i)\).

By Komlos' theorem (see \cite{komlos}) there exists a subsequence \(\lbrace m_{1,j} : j = 1,2,\ldots\rbrace\) such that
\[\lim\limits_{n \to +\infty}\frac{1}{n}\sum\limits_{j = 1}^{n}\nu_{s_{m_{1,j}},x_i}(f_1)\]
exists for \(\nu_i\)-almost every \(x_i\), and any further subsequence has the same property.

For each \(k = 1,2,\ldots\), using Komlos' theorem as above, we may define \(m_{k+1,1},m_{k+1,2},\ldots\) a subsequence of \(m_{k,1},m_{k,2},\ldots\) such that
\[\lim\limits_{n \to +\infty}\frac{1}{n}\sum\limits_{j = 1}^{n}\nu_{s_{m_{k+1,j}},x_i}(f_{k+1})\]
exists for \(\nu_i\)-almost every \(x_i\), and any further subsequence has the same property.

Letting \(t_n = s_{m_{n,n}}\) we have that
\[\lim\limits_{n \to +\infty}\frac{1}{n}\sum\limits_{j = 1}^{n}\nu_{t_j,x_i}(f_{k})\]
exists for \(\nu_i\)-almost every \(x_i\) and all \(k\).

For each \(x_i\) the the restriction of \(\lbrace f_k\rbrace\) to \(\pi^{-1}(x_i)\) is dense.  Since the space of probabilities on \(\pi^{-1}(x_i)\) is compact, this implies that there exist probabilities \(m_{x_i}\) such that
\[\lim\limits_{n \to +\infty}\frac{1}{n}\sum\limits_{j = 1}^{n}\nu_{t_j,x_i} = m_{x_i}\]
for \(\nu_i\)-almost every \(x_i\).

By lemma \ref{kakutaniargument} one has \(\lim\limits_{n \to +\infty}\nu_{t_n} = \nu\).  Therefore, for any continuous \(f\) by dominated convergence one has
\[\int m_{x_i}(f)d\nu_i(x_i) = \int \lim\limits_{n \to +\infty}\frac{1}{n}\sum\limits_{j = 1}^{n}\nu_{t_j,x_i}(f) d\nu_{x_i} = \lim\limits_{n \to +\infty} \frac{1}{n}\sum\limits_{j = 1}^{n}\nu_{t_j}(f)  = \nu(f),\]
from where \(\int m_{x_i} d\nu_i(x_i) = \nu\) and \(m_{x_i} = \nu_{x_i}\) for \(\nu_i\)-almost every \(x_i\).
\end{proof}

For each \(n\) let \(T_n\) be uniform \(\lbrace 1,\ldots,n\rbrace\) and independent from \((u_1,u_2,u_3,u_4)\), and let \(X_n = \rho_{\flags}(u_3,\nu_{T_n,F_i})\) and \(A_n = R_{T_n}A\).

\begin{claim}
 One has that \(X_n \eqd A_nX_n\).
\end{claim}
\begin{proof}
For any continuous function \(f:\flags(\R^d) \to \R\) one has
\begin{align*}
 \E{f(X_n)} &= \E{\E{f(X_n)\vert F_i}}
\\ &= \E{\frac{1}{n}\sum\limits_{k = 1}^n \nu_{t_k,F_i}(f)}
\\ &= \frac{1}{n}\sum\limits_{k = 1}^n  \int  \int f(x) d\nu_{t_k,x_i}(x) d\nu_i(x_i) 
\\ &= \frac{1}{n}\sum\limits_{k = 1}^n  \int  \int P_{t_k}f(x) d\nu_{t_k,x_i}(x) d\nu_i(x_i) 
\\ &= \frac{1}{n}\sum\limits_{k = 1}^n  \int \int  \int \int f(rax) d\lambda_{t_k,a x_i}(r) d\nu_{t_k,x_i}(x)d\nu_i(x_i) d\mu(a) 
\\ &= \E{\frac{1}{n}\sum\limits_{k = 1}^n  \int \int f(rAx) d\lambda_{t_k,AF_i}(r) d\nu_{t_k,F_i}(x)} 
\\ &= \E{\frac{1}{n}\sum\limits_{k = 1}^n  \int f(R_{t_k}Ax)  d\nu_{t_k,F_i}(x)} 
\\ &= \E{f(A_nX_n)\vert A,F_i} 
\\ &= \E{f(A_nX_n)}.
\end{align*} 
\end{proof}

Since \(T_n\) converges in distribution to \(0\) there exists a subsequence such that \(\lim\limits_{k \to +\infty}T_{n_k} = 0\) almost surely.  We fix such a subsequence from now on.

\begin{claim}
 The conditional distribution of \((A,A_{n_k}X_{n_k})\) given \(AF_i\) converges almost surely to the conditional distribution of \((A,F)\) given \(AF_i\).
\end{claim}
\begin{proof}
It suffices to show that for all bounded and uniformly continuous \(f\) one has
\[\lim\limits_{k \to +\infty}\E{f(A,A_{n_k}X_{n_k})\vert AF_i} = \E{f(A,AF)\vert AF_i},\]
almost surely.

The distance between \(A_{n_k}X_{n_k} = R_{T_{n_k}}AX_{n_k}\) and \(AX_{n_k}\) goes to \(0\) almost surely.  Therefore, since \(f\) is uniformly continuous, one has
\[\lim\limits_{k \to +\infty}|f(A,A_{n_k}X_{n_k}) - f(A,AX_{n_k})| = 0,\]
almost surely.

Because \(f\) is bounded, by dominated convergence, the limit above also holds in the \(L^1\) sense, and therefore
\[\lim\limits_{k \to +\infty}\E{f(A,A_{n_k}X_{n_k}) - f(A,AX_{n_k}) \vert AF_i} = 0,\]
almost surely.

Noticing that \(\sigma(AF_i) \subset \sigma(A,F_i) = \sigma(A,AF_i)\), we now calculate 
\begin{align*}
\lim\limits_{k \to +\infty}\E{f(A,A_{n_k}X_{n_k})\vert AF_i} &= \lim\limits_{k \to +\infty}\E{f(A,AX_{n_k})\vert AF_i} 
\\ &= \lim\limits_{k \to +\infty}\E{\E{ f(A,AX_{n_k})\vert A, F_i} \vert AF_i}
 \\ &= \lim\limits_{k \to +\infty}\E{\frac{1}{n_k}\sum\limits_{j=1}^{n_k} \int f(A,Ax)d\nu_{t_j,F_i}(x) \vert AF_i}
 \\ &= \E{ \int f(A,Ax)d\nu_{F_i}(x) \vert AF_i}
 \\ &= \E{ \E{f(A,AF) \vert A,F_i} \vert AF_i}
 \\ &= \E{f(A,AF) \vert AF_i} 
\end{align*}
where we have used the almost sure convergence of \(\frac{1}{n_k}\sum\limits_{j = 1}^{n_k}\nu_{t_j,F_i}\) to \(\nu_{F_i}\) and boundedness of \(f\) to move the limit inside the expected value in the third to last step.
\end{proof}

In view of the above claim by, lemma \ref{entropyandmutualinformationlemma}, proposition \ref{semicontinuityproposition}, and Fatou's lemma we have
\begin{align*}
 \kappa_i &= \E{I(A,AF|AF_i)} 
 \\ &\le \E{\liminf\limits_{k \to +\infty}I(A,A_{n_k}X_{n_k}|AF_i)}
 \\ &\le \liminf\limits_{k \to +\infty}\E{I(A,A_{n_k}X_{n_k}|AF_i)}.
\end{align*}

By monotonocity and the chain rule (Proposition \ref{cocycleproposition}) we continue
\begin{align*}
\liminf\limits_{k \to +\infty}\E{I(A,A_{n_k}X_{n_k}|AF_i)}  &\le \liminf\limits_{k \to +\infty}\E{I(A,(T_{n_k},A_{n_k}X_{n_k})\vert AF_i}
\\ &= \liminf\limits_{k \to +\infty}\E{I(A,T_{n_k}\vert AF_i) + I(A,A_{n_k}X_{n_k}\vert T_{n_k},AF_i)},
\end{align*}

Because \(T_{n_k}\) is independent from \(A\) and \(AF_i\) we have
\[\liminf\limits_{k \to +\infty}\E{I(A,T_{n_k}\vert AF_i) + I(A,A_{n_k}X_{n_k}\vert T_{n_k},AF_i)} = \liminf\limits_{k \to +\infty}\E{I(A,A_{n_k}X_{n_k}\vert T_{n_k},AF_i)}.\]

And, finally, by monotonicity of mutual information
\[\liminf\limits_{k \to +\infty}\E{I(A,A_{n_k}X_{n_k}\vert T_{n_k},AF_i)} \le \liminf\limits_{k \to +\infty}\E{I((R_{n_k},A),A_{n_k}X_{n_k}\vert T_{n_k},AF_i)}.\]

We conclude the proof by establishing the following:
\begin{claim}
In the above context one has:
 \[\liminf\limits_{k \to +\infty}\E{I((R_{n_k},A),A_{n_k}X_{n_k}\vert T_{n_k},AF_i)} = \chi_i - \chi_{i+1}.\]
\end{claim}
\begin{proof}
We first claim that the the conditional distribution of \(A_{n_k}X_{n_k}\) given \(\F_1 = \sigma(T_{n_k}, A_{n_k}F_i) = \sigma(T_{n_k},AF_i)\) is \(\nu_{T_{n_k},AF_i}\).

Since \(AF_i\) has distribution \(\nu_i\) and is independent from \(T_{n_k}\) it suffices to prove that the conditional distribution of \((A_{n_k}X_{n_k},AF_i)\) given \(T_{n_k}\) coincides with that of \((X_{n_k},F_i)\) given \(T_{n_k}\).

Since \(F_i = \pi(X_{n_k})\) and \(AF_i = \pi(A_{n_k}X_{n_k})\), we only need to verify that the conditional distribution of \(X_{n_k}\) given \(T_{n_k}\) (which is \(\nu_{T_{n_k}}\) coincides with that of \(A_{n_k}X_{n_k}\) given \(T_{n_k}\).  This follows immediately since \(\nu_{T_{n_k}}\) is \(P_{T_{n_k}}^*\)-invariant.

Now notice that the conditional distribution of \(X_{n_k}\) given \(\F_2 = \sigma((R_{T_{n_k}},A),\F_1) = \sigma(R_{T_{n_k}},A,T_{n_k},AF_i) = \sigma(R_{T_{n_k}},A,T_{n_k},F_i)\) is also \(\nu_{T_{n_k},F_i}\).   
This implies that the conditional distribution of \(A_{n_k}X_{n_k}\) given \(\F_2\) is \(R_{n_k}A\nu_{T_{n_k},F_i}\).

Let \(m\) denote the conditional distribution of \((R_{n_k},A)\) given \(\F_1\).
 
 We have shown that the joint distribution of \((R_{n_k},A),A_{n_k}X_{n_k}\) given \(\F_1\) has projections \(m\) and \(\nu_{T_{n_k},AF_i}\), while its disintegration onto the factor \(m\) has conditional measures \(R_{n_k}A\nu_{T_{n_k},F_i}\).
 
 Applying the Gelfand-Yaglom-Perez theorem this yields
 \[\E{I((R_{n_k},A),A_{n_k}X_{n_k}\vert T_{n_k},AF_i)} = \E{\frac{dR_{n_k}A\nu_{T_{n_k},AF_i}}{d\nu_{T_{n_k},AF_i}}(R_{n_k}AX_{n_k})}.\]
 
 Let \(S_{n,i}\) be the \(i\)-dimensional subspace of \(X_{n}\) and \(\varphi_{t_n,F_i} = \frac{d\nu_{t_n,F_i}}{d\eta_{F_i}}\).
 Using lemma \ref{grassmannjacobian} we obtain
\begin{align*}
 \E{\frac{dR_{n_k}A\nu_{T_{n_k},AF_i}}{d\nu_{T_{n_k},AF_i}}(R_{n_k}AX_{n_k})} &= \E{\log\left(\frac{|\det_{S_{n_k,i}}(R_{n_k}A)|^2}{|\det_{S_{i-1}}(R_{n_k}A)||\det_{S_{i+1}}(R_{n_k}A)|}\right)} 
 \\ &+ \E{\log\left(\frac{\varphi_{T_{n_k},F_i}(X_{n_k})}{\varphi_{T_{n_k},AF_i}(A_{n_k}X_{n_k})}\right)}
 \\ &= \E{\log\left(\frac{|\det_{S_{n_k,i}}(R_{n_k}A)|^2}{|\det_{S_{i-1}}(R_{n_k}A)||\det_{S_{i+1}}(R_{n_k}A)|}\right)}
 \end{align*}
where for the last equality we have used that \((T_{n_k},X_{n_k},\pi(X_{n_k})) = (T_{n_k},X_{n_k},F_i)\) has the same distribution as \((T_{n_k},R_{n_k}AX_{n_k}, \pi(R_{n_k}A)) = (T_{n_k},R_{n_k}AX_{n_k},AF_i)\).

Since \(R_{n_k}\) is an orthogonal transformation the determinants of \(R_{n_k}A\) and \(A\) coincide on all subspaces.  Therefore the right hand side above is equal to
\[\E{\log\left(\frac{|\det_{S_{n_k,i}}(A)|^2}{|\det_{S_{i-1}}(A)||\det_{S_{i+1}}(A)|}\right)}.\]

Since \(X_{n_k}\) converges in distribution to \(F\) we have that \(S_{n_k,i}\) converges in distribution to \(S_i\) the \(i\)-dimensional subspace of \(F\).
Because the logarithm of the determinant of \(A\) on any subspace is bounded between constant multiples of \(\log(\sigma_1(A))\) and \(\log(\sigma_d(A))\) both of which are integrable, we can pass to the limit (e.g. using dominated convergence after replacing \(S_{n_k,i}\) by a sequence with the same individual distributions but which converges almost surely, see \cite[Theorem 6.7]{billingsley})  obtaining
\[\lim\limits_{k \to +\infty}\E{\log\left(\frac{|\det_{S_{n_k,i}}(A)|^2}{|\det_{S_{i-1}}(A)||\det_{S_{i+1}}(A)|}\right)} = \E{\log\left(\frac{|\det_{S_{i}}(A)|^2}{|\det_{S_{i-1}}(A)||\det_{S_{i+1}}(A)|}\right)}.\]

Finally since \(\E{|\det_{S_j}(A)|} = \chi_1 + \cdots + \chi_j\) for \(j = 1,\ldots, d\), one obtains
\[\E{\log\left(\frac{|\det_{S_{i}}(A)|^2}{|\det_{S_{i-1}}(A)||\det_{S_{i+1}}(A)|}\right)} = \chi_i - \chi_{i+1},\]
which concludes the proof.
\end{proof}

\part{Exact dimensionality and dimension of conditional probabilities}

In this part of the article we will prove Theorem \ref{dimensiontheorem}.  We now specify notation and context that will be used throughout.

Recall that \(\mu\) is a probability on \(\GL(\R^d)\) with respect to which the logarithm of all singular values are integrable and \(\nu\) is a \(\mu\)-stationary probability on \(\Flags(\R^d)\).

A dimension \(i \in \lbrace 1,\ldots, d-1\rbrace\) is fixed throughout, \(\nu_i\) is the projection of \(\nu\) on the space \(\flags_i(\R^d)\) of incomplete flags missing their \(i\)-dimensional subspace.    It is assumed that \(\nu\) is the unique stationary probability with projection \(\nu_i\).

A disintegration \(F_i \mapsto \nu_{F_i}\) of \(\nu\) with respect to \(\nu_i\) is fixed (so \(\nu = \int \nu_{F_i} d\nu_i(F_i)\)).

We consider an i.i.d. sequence \((A(n))_{n \in \Z}\) with common distribution \(\mu\) and a stationary sequence of random random flags \((F(n))_{n \in \Z}\) with common distribution \(\nu\) such that
\[A(n+k)\cdots A(n)F(n) = F(n+k)\]
for all \(n \in \Z\) and \(k \ge 0\).   We will use \(S_j(n)\) for the \(j\)-dimensional subspace of the flag \(F(n)\) and \(F_i(n)\) as before for the incomplete flag obtained by removing the subspace \(S_i(n)\).

By hypothesis \(\nu\) is ergodic (i.e. extremal among stationary probabilities) this implies that the stationary sequence \(((F(n),A(n)))_{n \in \Z}\) is ergodic.

As before, Lyapunov exponents \(\chi_1,\ldots,\chi_d\) are defined by the equations
\[\chi_1 + \cdots + \chi_j = \E{\log\left(\left|\det_{S_i(n)}(A(n))\right|\right)}.\]

By Theorem \ref{gaptheorem} one has \(A\nu_{F_i(n)} \ll \nu_{F_i(n+1)}\) almost surely and 
\[0 \le \kappa_i = \E{\log\left(\frac{dA\nu_{F_i(n)}}{d\nu_{F_i(n+1)}}(F(n+1))\right)} \le \chi_i - \chi_{i+1}.\]

We assume from now on that \(\kappa_i > 0\).

\section{Non-atomicity of conditional measures}

Our first step in the proof of Theorem \ref{dimensiontheorem} is that \(\nu_{F_i(n)}\) is almost surely non-atomic (i.e. all points have measure zero).
 \begin{lemma}\label{nonatomicitylemma}
 Almost surely \(\nu_{F_i(n)}\) is non-atomic for all \(n\).
 \end{lemma}
\begin{proof}
By ergodicity and one has
\[\kappa_i = \lim\limits_{n \to +\infty}\frac{1}{n}\log\left(\frac{dA(n-1)\cdots A(0)\nu_{F_i(0)}}{d\nu_{F_i(n)}}(F(n))\right),\]
almost surely.

Suppose for the sake of contradiction that \(\Pof{\nu_{F_i(0)}(F(0)) > 0} > 0\).  Conditioning on this event the equation above becomes
\[\kappa_i = \lim\limits_{n \to +\infty}\frac{1}{n}\log\left(\frac{\nu_{F_i(0)}(F(0))}{\nu_{F_i(n)}(F(n))}\right).\]

However, by Poincaré recurrence \(\nu_{F_i(n)}(F(n))\) is recurrent almost surely (i.e. almost surely there exists a subsequence such that \(\lim\limits_{k}\nu_{F_i(n_k)}(F(n_k)) = \nu_{F_i(0)}(F(0))\)).   This implies that \(\kappa_i = 0\) which contradicts the hypothesis that \(\kappa_i > 0\).  Hence, \(\nu_{F_i(0)}(F(0)) = 0\) almost surely, as claimed.
\end{proof}

\section{The multiplicative ergodic theorem}

From Theorem \ref{gaptheorem} and the hypothesis that \(\kappa_i > 0\) one obtains that \(\chi_i > \chi_{i+1}\).  We will now apply the multiplicative ergodic theorem of \cite{oseledets} to the mappings induced by the sequence \(A(n)\) between the quotient spaces \(S_{i+1}(n)/S_{i-1}(n)\) to obtain the following result:

\begin{lemma}\label{oseledetslemma}
 Almost surely for each \(n\) one has
 \[\lim\limits_{k \to +\infty}\frac{1}{k}\log\left(|\det_{S_i(n)}(A(n+k-1)\cdots A(n))|\right) = \chi_1 + \cdots + \chi_{i-1} + \chi_{i},\]
 and there exists a unique \(i\)-dimensional subspace \(S_i'(n)\) containing \(S_{i-1}(n)\) and contained in \(S_{i+1}(n)\) such that
 \[\lim\limits_{k \to +\infty}\frac{1}{k}\log\left(|\det_{S_i'(n)}(A(n+k-1)\cdots A(n))|\right) = \chi_1 + \cdots + \chi_{i-1} + \chi_{i+1}.\]
 
 Furthermore, \(S_i(n)\) and \(S_i'(n)\) are conditionally independent given \(F_i(n)\), and \(S_i(n) \neq S_i'(n)\) almost surely.
 
 Finally, the logarithm of the angle between the projections of \(S_i(n)\) and \(S_i'(n)\) to \(S_{i+1}(n)/S_{i-1}(n)\) is \(o(|n|)\) when \(n \to \pm \infty\). 
\end{lemma}
\begin{proof}
 For each \(n\) consider the quotient space \(V(n) = S_{i+1}(n)/S_{i-1}(n)\) with the induced inner product coming from \(\R^d\),  let \(E^u(n)\) be the one-dimensional subspace in \(V(n)\) which is the projection of \(S_i(n)\), and let \(T(n): V(n) \to V(n+1)\) be mapping induced by \(A(n)\). 
 
 Notice that almost surely each \(V(n)\) is isometric to \(\R^2\) with the usual inner product.  Furthermore the random sequence
\[\cdots \stackrel{T(n-1)}{\mapsto} (V(n),E^u(n)) \stackrel{T(n)}{\mapsto} (V(n+1),E^u(n+1)) \stackrel{T(n+1)}{\mapsto} \cdots\]
is stationary and ergodic.

 One has
 \[\E{\log\left(\left|\det_{E^u(n)}(T_n)\right|\right)} = \chi_i\]
 which implies by Birkhoff's theorem that almost surely
 \[\lim\limits_{k \to +\infty}\frac{1}{k}\log\left(\|T(n-k)^{-1}\cdots T(n-1)^{-1}v\|\right) = -\chi_i\]
 and
 \[\lim\limits_{k \to +\infty}\frac{1}{k}\log\left(\|T(n+k-1)\cdots T(n)v\|\right) = \chi_i\]
 for all \(v \in E^u(n) \setminus \lbrace 0\rbrace\).

 On the other hand
 \[\E{\log\left(\left|\det(T_n)\right|\right)} = \chi_i + \chi_{i+1}.\]
 which implies that almost surely
 \[\lim\limits_{k \to +\infty}\frac{1}{k}\log\left(\left|\det(T(n+k-1)\cdots T(n))\right|\right) = \chi_i + \chi_{i+1}.\]
 
 By hypothesis \(\kappa_i > 0\) which implies by Theorem \ref{gaptheorem} that \(\chi_i > \chi_{i+1}\).  Hence, one obtains from the multiplicative ergodic theorem of \cite{oseledets} that almost surely
 \[E^u(n) = \lbrace 0 \rbrace \cup \lbrace v \in V(n): \lim\limits_{k \to +\infty}\frac{1}{k}\log\left(\|T(n-k)^{-1}\cdots T(n-1)^{-1}v\|\right) = -\chi_i\rbrace\]
 and 
 \[E^s(n) = \lbrace 0 \rbrace \cup \lbrace v \in V(n): \lim\limits_{k \to +\infty}\frac{1}{k}\log\left(\|T(n+k-1)\cdots T(n)v\|\right) = \chi_{i+1}\rbrace,\]
 are complementary one-dimensional subspaces, and the angle between them is \(e^{o(n)}\).

 From the equations above it follows that \(E^u(n)\) is \(\sigma(F_i(n),A(n-1),A(n-2),\ldots)\)-measurable, while \(E^s(n)\) is \(\sigma(F_i(n),A(n),A(n+1),\ldots)\)-measurable.  Since \(F_i(n)\) is \(\sigma(A(n-1),A(n-2),\ldots)\)-measurable one has that \((A(n-1),A(n-2),\ldots)\) and \((A(n),A(n+1),\ldots)\) are conditionally independent given \(F_i(n)\).   
 In particular, conditioned on \(F_i(n)\) one has that \(E^u(n)\) and \(E^s(n)\) are independent.

 Setting \(S_i'(n)\) to be the subspace in \(S_{i+1}(n)\) which projects to \(E^s(n)\) in \(S_{i+1}(n)/S_{i-1}(n)\) one obtains the desired result. 
\end{proof}

\section{Proof of Theorem \ref{dimensiontheorem}}

\subsection{Random circle diffeomorphisms}

We fix from now on a Borel measurable projection from \(\flags_i(\R^d)\) to \(\R^2\) which consists of mapping \(S_{i+1}/S_{i-1}\) to \(\R^2\) isometrically (where \(S_j\) denotes the \(j\)-dimensional subspace of the flag).  Furthermore we fix an isometry between the unit circle \(S^1\) with the usual arc-length distance scaled by one half \(\dist\), and the space of one-dimensional subspaces of \(\R^2\) with the distance given by the angle.  The composition of these mappings will be used to identify each fiber of the projection from \(\Flags(\R^d)\) to \(\Flags_i(\R^d)\) with the unit circle.  Equivalently, given an incomplete flag \(F_i = (S_0,\ldots,S_d)\) we have chosen an isometry from the projective space of \(S_{i+1}/S_{i-1}\) to the unit circle, and therefore each \(i\)-dimensional subspace between \(S_{i-1}\) and \(S_{i+1}\) corresponds to a point on the unit circle.

With these identifications let \(\F_n = \sigma(F_i(n))\), \(\nu_n\) be  be the projection of \(\nu_{F_i(n)}\) to \(S^1\), \(x_n\) be the projection of \(S_i(n)\) to \(S^1\), \(y_n\) be the projection of \(S_i'(n)\) (given by Lemma \ref{oseledetslemma}) to \(S^1\), \(T_n\) the diffeomorphism of \(S^1\) obtained by projecting the action of \(A(n)\) between \(S_{i+1}(n)/S_{i-1}(n)\) and \(S_{i+1}(n+1)/S_{i-1}(n+1)\), and for convenience let \(\kappa = \kappa_i\) and \(\chi = \chi_i - \chi_{i+1}\).   Finally, we let \(\eta\) be the rotationally invariant probability on the unit circle.

The proof of Theorem \ref{dimensiontheorem} will proceed as follows:  We will construct a sequence of random intervals \(I_n\) containing \(x_n\) and such that \(T_{-1}\circ \cdots \circ T_{-n}(I_{-n})\) is roughly of size \(e^{-\chi n}\).  We will then show that \(\nu_0(T_{-1}\circ \cdots \circ T_{-n}(I_{-n}))\) is roughly \(e^{-\kappa n}\).  These two facts will yield that the local dimension of \(\nu_0\) at \(x_0\) is almost surely \(\kappa/\chi\) so that in particular that \(\nu_0\) is exact dimensional.  

A few technical issues arise which we have concealed with the word `roughly' in the previous paragraph.  For example, the estimates for the measure of the intervals will hold only for some values of  \(n\), but these values are sufficiently dense to imply the needed dimension estimates.

We begin with a simple consequence of lemma \ref{oseledetslemma}.
\begin{proposition}\label{simplesizeprop}
 Let \(k \in \Z\) and \(\epsilon \in (0,1)\) be fixed and let \(I = S^{1} \setminus B_{\epsilon \dist(x_k,y_k)}(x_k)\).
 
 Then the length of \(T_{k-n}^{-1}\circ \cdots \circ T_{k-1}^{-1}(I)\) converges to \(0\) exponentially quickly when \(n \to +\infty\).
\end{proposition}
\begin{proof}
 The interval \(I\) corresponds to a cone \(C\) of one dimensional subspaces in \(S_{i+1}(k)/S_{i-1}(k)\) whose angle (with respect to the standard inner product inherited from \(\R^d\)) with the projection \(E^u\) of \(S_i(k)\) is larger than \(\epsilon\) times the angle between \(E^u\) and the projection \(E^s\) of \(S_i'(k)\).

 By lemma \ref{oseledetslemma}, under the action the linear mapping \(L_n\) corresponding to \(T_{k-n}^{-1}\circ \cdots \circ T_{k-1}^{-1}\) the norm of vectors in \(E^u\) are multiplied by a factor of \(e^{-\chi_i n + o(n)}\) while those in \(E^s\) are multiplied by a factor of \(e^{-\chi_{i+1} n + o(n)}\).

 We fix on the domain of \(L_n\) the inner product for which the norm on \(E^u,E^s\) coincides with the standard one, but for which these subspaces are orthogonal.
 
 Similarly on the range of \(L_n\) we pick the inner product where \(L_nE^u,L_nE^s\) are orthogonal and the restriction of the norm on both subspaces coincides with the usual one.
 
 With respect to these inner products the angle between any two subspaces in \(C\) decreases by a factor of \(e^{-(\chi_i - \chi_{i+1})n + o(n)}\) under \(L_n\).
 
 However, once again by lemma \ref{oseledetslemma}, the angle between \(L_nE^u,L_nE^s\) is \(e^{o(n)}\) for the standard inner product.  This implies that, measured with the standard inner product the angle between any two subspaces of \(C\) decreases by the same factor up to a multiplicative \(e^{o(n)}\).
\end{proof}

\subsection{Stationary intervals}

We now construct the sequence of intervals that will be used in our argument.  The key points for what follows are that: the construction is stationary, the intervals contain \(x_n\) but not \(y_n\), their size is controlled by \(\dist(x_n,y_n)\), and frequently \(\nu_n(I_n)\) is not close to zero.
\begin{lemma}[Stationary intervals]\label{stationaryintervalslemma}
Setting 
\[I_n = S^1 \setminus B_{\frac{1}{2}\dist(x_n,y_n)}(y_n),\]
one has \(\Pof{\nu_n(I_n) \ge 1/2} \ge 1/2\) for all \(n\). 
\end{lemma}
\begin{proof}
Since almost surely \(\nu_{n}\) is non-atomic there is a smallest positive radius \(r_n\) such that \(\nu_n(B_{r_n}(y_n)) = \nu_n(S^1 \setminus B_{r_n}(y_n)) = 1/2\).

By lemma \ref{oseledetslemma}, conditioned on \(\F_n\) one has that \(x_n\) has distribution \(\nu_n\) and is independent from \(r_n\) and \(y_n\).  Therefore \(\P(x_n \in B_{r_n}(y_n)|\F_n) = \nu_n(B_{r_n}(y_n)) = 1/2\) and taking expected value \(\P(x_n \in B_{r_n}(y_n)) = 1/2\).

In the event that \(x_n \in B_{r_n}(y_n)\) one has that \(S^1 \setminus B_{r_n}(y_n)   \subset I_n\) and therefore that \(\nu_n(I_n) \ge 1/2\).  This proves the claim.
\end{proof}

What remains is to estimate the size and \(\nu_0\) probability of the sequence \(T_{-1}\circ \cdots \circ T_{-n}(I_{-n})\).

\begin{figure}[H]
\centering
\includegraphics[scale=0.5]{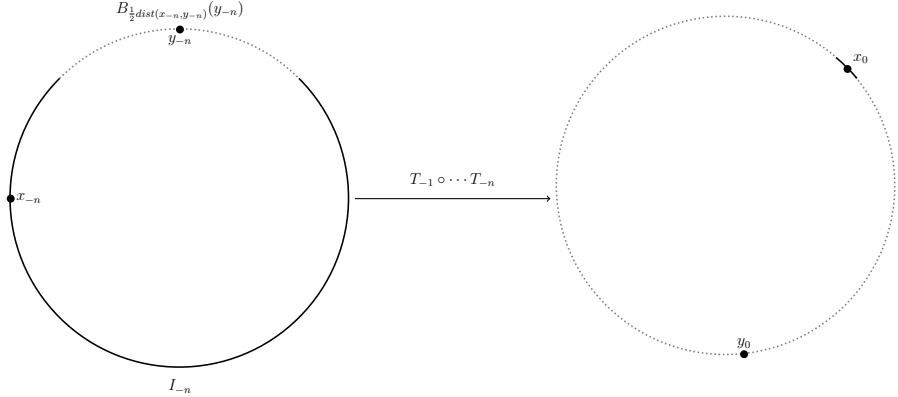}
\caption{\label{fig1}For large \(n\) the transformation \(T_{-1}\circ \cdots \circ T_{-n}\) contracts the large interval \(I_{-n}\) to an interval of size roughly \(e^{-\chi n}\) (see Lemma \ref{lengthlemma}).  
With frequency at least \(1/2\) the \(\nu_0\)-measure of the image interval is roughly \(e^{-\kappa n}\) (see lemmas \ref{stationaryintervalslemma} and \ref{probabilitylemma}).}
\end{figure}

\subsection{Length of distinguished intervals}

The point of what follows is that the intervals \(T_{-1}\circ \cdots \circ T_{-n}(I_{-n})\) contain \(x_0\) and are roughly of size \(e^{-\chi n}\).

We will use the following result which is essentially Maker's theorem \cite[Theorem 1]{maker1940} or \cite[Theorem 1]{breiman1957}.
\begin{theorem*}[Maker's theorem]
 Let \((X_{n,k})_{k,n \in \Z}\) be a family of random variables which is stationary in the sense that its distribution equals that of \((Y_{k,n})_{k,n \in \Z}\) where \(Y_{k,n} = X_{k+1,n}\).
 
 Suppose that the limit \(X_k = \lim\limits_{n \to +\infty}X_{k,n}\) exists almost surely and that \(\E{\sup\limits_{n}|X_{k,n}|} < +\infty\) for all (or equivalently due to stationary, for some) \(k\).
 
 Then \(\lim\limits_{n \to +\infty}\frac{1}{n}\sum\limits_{k = 0}^{n-1}X_{-k,n-k} = \lim\limits_{n \to +\infty}\frac{1}{n}\sum\limits_{k = 0}^{n-1}X_{-k}\) almost surely.
\end{theorem*}
\begin{proof}
By Birkhoff's ergodic theorem
\[X = \lim\limits_{n \to +\infty}\frac{1}{n}\sum\limits_{k = 0}^{n-1}X_{-k},\]
exists almost surely and \(\E{X} = \E{X_0}\) is finite.

Following \cite[Theorem 1]{breiman1957} we write
\[\frac{1}{n}\sum\limits_{k = 0}^{n-1}X_{-k,n-k} = \frac{1}{n}\sum\limits_{k = 0}^{n-1}X_{-k} + \frac{1}{n}\sum\limits_{k = 0}^{n-1}(X_{-k,n-k}-X_{-k}).\]

The first term converges to \(X\) almost surely.  Letting \(Y_n\) be the second term notice that for any fixed \(N\) we have
\[\limsup\limits_{n \to +\infty}|Y_n| = \limsup\limits_{n \to +\infty}\left|\frac{1}{n}\sum\limits_{k = 0}^{n-1}(X_{-k,n-k}-X_{-k})\right| \le \lim\limits_{n \to +\infty}\frac{1}{n}\sum\limits_{k = 0}^{n-1}\sup\limits_{n \ge N}|X_{-k,n}-X_{-k}| = Z_N,\]
where the limit defining \(Z_N\) exists almost surely and satisfies \(\E{Z_N} = \E{\sup\limits_{n \ge N}|X_{0,n}-X_0|} < +\infty\) by Birkhoff's ergodic theorem.

Since \(\sup\limits_{n \ge N}|X_{0,n}-X_0|\) decreases monotonely to \(0\) we obtain
\[\E{\limsup\limits_{n \to +\infty}|Y_n|} \le \lim\limits_{N \to +\infty}\E{Z_N} = 0,\]
so that \(\limsup\limits_{n \to +\infty}|Y_n| = 0\) almost surely.
\end{proof}

\begin{lemma}[Length of distinguished intervals]\label{lengthlemma}
 For all \(\epsilon > 0\) almost surely one has
 \[B_{r_n}(x_0) \subset  T_{-1}\cdots T_{-n}I_{-n} \subset B_{R_n}(x_0)\]
 for all \(n\) large enough, where \(r_n = \exp(-(\chi+\epsilon)n)\) and \(R_n = \exp(-(\chi-\epsilon)n)\).
\end{lemma}
\begin{proof} 
 Recall that \(\eta\) denotes the rotationally invariant probability on the unit circle \(S^1\).

 By Lemma \ref{grassmannjacobian} one has
 \[\chi = \E{\log\left(\frac{dT_{k-1}\eta}{d\eta}(x_k)\right)},\]
 for all \(k\).
 
 For each \(n\) let \(J_n\) be the connected component of \(I_n \setminus \lbrace x_n\rbrace\) which is counter-clockwise from \(x_n\) and define
 \[X_{k,n} = \log\left(\frac{\eta(T_{k-2}\circ \cdots \circ T_{k-n}(J_{k-n}))}{\eta(T_{k-1}\circ \cdots \circ T_{k-n}(J_{k-n}))}\right),\]
 and
 \[X_k = \log\left(\frac{dT_{k-1}\eta}{d\eta}(x_k)\right).\]
    
By lemma \ref{oseledetslemma} one has that \(S^1 \setminus J_{n-k}\) contains a ball of radius \(e^{o(n)}\) centered at \(y_{n-k}\).   In view of proposition \ref{simplesizeprop} this implies that for all \(\epsilon \in (0,1)\) almost surely eventually \(T_{k-1}\circ \cdots \circ T_{k-n}(J_{k-n}) \subset B_{\epsilon \dist(x_k,y_k)}(x_k)\).
This shows that almost surely the length of \(T_{k-1}\circ \cdots \circ T_{k-n}(J_{k-n})\) goes to zero when \(n \to +\infty\) and therefore one one has \(\lim\limits_{n \to +\infty}X_{k,n} = X_k\) almost surely for all \(k\).
 
 Notice that for each \(x \in S^1\) one has, again by lemma \ref{grassmannjacobian}, that
 \[\frac{dT_{k-1}\eta}{d\eta}(x) = \frac{|\det_{S}(A(k-1))|^2}{|\det_{S_{i-1}(k-1)}(A(k-1))||\det_{S_{i+1}(k-1)}(A(k-1))|}\]
 for some \(i\)-dimensional subspace \(S\) between \(S_{i-1}(k-1)\) and \(S_{i+1}(k-1)\).
 
 In particular this implies that
 \(\min_x \log\left(\frac{dT_{k-1}\eta}{d\eta}(x)\right)\)
 and 
 \(\max_x \log\left(\frac{dT_{k-1}\eta}{d\eta}(x)\right)\)
 have finite expectation since they are controlled by the logarithms of singular values of \(A(k-1)\).
 
 This yields that \(\sup\limits_{n \to +\infty}|X_{k,n}|\) has finite expectation for all \(k\).
 
Applying Maker's theorem we obtain
 \begin{align*}
\chi &= \lim\limits_{n \to +\infty}\frac{1}{n}\sum\limits_{k = 0}^{n-1} \log\left(\frac{dT_{-k-1}\eta}{d\eta}(x_{-k})\right)
\\ &=  \lim\limits_{n \to +\infty}\frac{1}{n}\sum\limits_{k = 0}^{n-1} \log\left(\frac{\eta(T_{-k-2}\circ \cdots \circ T_{-n}(J_{-n}))}{\eta(T_{-k-1}\circ \cdots \circ T_{-n}(J_{-n}))}\right)
\\ &=  \lim\limits_{n \to +\infty}\frac{1}{n} \log\left(\frac{\eta(J_{-n})}{\eta(T_{-1}\circ \cdots \circ T_{-n}(J_{-n}))}\right).
\end{align*}
  
 Finally, since \(\eta(J_{-n}) = e^{o(n)}\) when \(n \to +\infty\) by Lemma \ref{oseledetslemma} one obtains:
\[\lim\limits_{n \to +\infty}\frac{1}{n} \log\left(\eta(T_{-1}\circ \cdots \circ T_{-n}(J_{-n})\right) = -\chi.\]

The same argument shows that \(\eta(T_{-1} \circ \cdots T_{-n}(I_{-n} \setminus J_{-n})) = e^{-\chi n}\) which establishes the claims.

 \end{proof}

\subsection{Probability of distinguished intervals}

We will now essentially repeat the argument of the previous subsection replacing the rotationally invariant probability measure (which is equivalent to length up to a factor) with the random probabilities \(\nu_n\).

In this case one wishes to replace (in the ergodic averages) the terms of the form \(\frac{dT_{k-1}\nu_{k-1}}{d\nu_{k}}(x_k)\) with approximating terms calculated using the intervales \(I_n\).  Almost sure convergence of the approximating terms boils down to the theorem on differentiation of measures.  However, the integrability of the supremum of the approximating terms is more subtle.

The issue is that the singular values of \(A(k-1)\) do not directly control the maximum and minimum of \(\frac{dT_{k-1}\nu_{k-1}}{d\nu_{k}}(x)\) on the circle.   In fact, this density may be unbounded with positive probability.    Instead, control of the approximation comes from the \(x\log(x)\)-integrability of the density with respect to \(\nu_k\) which follows from the fact that \(\kappa < +\infty\) (that is Theorem \ref{gaptheorem}).

\subsubsection{Orlicz regularity and a maximal inequality\label{sectionorlicz}}

For each \(k\) let \(f_k(x) = \frac{dT_{k-1}\nu_{k-1}}{d\nu_k}(x)\) and notice that it is \(\sigma(\F_{k-1},\F_k,T_{k-1})\)-measurable.

Notice that \(x_{k-1}\) and \(A(k-1)\) are independent conditioned on \(\F_{k-1}\).   Since the conditional distribution of \(x_{k-1}\) given \(\F_{k-1}\) is \(\nu_{k-1}\) one obtains that the distribution of \(x_{k} = T_{k-1}(x_{k-1})\) 
conditioned on \(\sigma(\F_{k-1},A(k-1))\) has density \(f_k\) with respect to \(\nu_k\).   Since this conditional distribution is \(\sigma(\F_{k-1},\F_k,T_{k-1})\)-measurable and \(\sigma(\F_{k-1},\F_k,T_{k-1}) \subset \sigma(\F_{k-1},A(k-1))\) one obtains
that the conditional distribution of \(x_k\) given \(\sigma(\F_{k-1},\F_k,T_{k-1})\) has density \(f_k\) with respect to \(\nu_k\).  Therefore, 
\begin{align*}
\kappa &= \E{\log\left(f_k(x_k)\right)} = \E{\E{\log\left(f_k(x_k)\right)|\F_{k-1},\F_k,T_{k-1}}}
\\ &= \E{\int f_k(x)\log(f_k(x))d\nu_k(x)}.
\end{align*}

In particular \(f_k\log(f_k)\) is almost surely integrable with respect to \(\nu_k\).  In other words, \(f_k\) almost surely belongs to an Orlicz space which is slightly smaller than \(L^1(\nu_k)\) and the expected value of the corresponding Orlicz norm is finite.  This fact, which follows from the finiteness of \(\kappa\) given by Theorem \ref{gaptheorem}, will allow us to control the maximal function of \(f_k\).

We define the maximal function of a function \(f:S^1 \to \R\) with respect to a probability \(\lambda\) as 
\[M_\lambda f(x) = \sup\limits_{x \in I} \frac{1}{\nu(I)}\int\limits_{I}|f(y)|d\lambda(y)\]
where the supremum is over all intervals containing \(x\).

We will need the following maximal inequality the proof of which is adapted from the proof of \cite[Theorem 1]{stein1970}.
\begin{lemma}[Maximal inequality]\label{maximalinequality}
There exists a constant \(C > 0\) such that for any probability \(\lambda\) on \(S^1\) and any \(\lambda\)-integrable function \(f\) one has
\[t \lambda\left(\lbrace x: M_\lambda f(x) > t\rbrace\right) \le C\int |f|\one{\lbrace |f| > t/2\rbrace}d\lambda\]
for all \(t > 0\).
\end{lemma}
\begin{proof}
 Given \(\lambda\),\(f\), and \(t\) consider a compact set \(K \subset \lbrace M_\lambda f >  t\rbrace\) such that
 \[\lambda(\lbrace M_\lambda f > t\rbrace) \le 2\lambda(K).\]
 
 By definition, each point in \(K\) belongs to an interval \(I\) such that
 \[t\lambda(I) < \int |f|\one{I}d\lambda.\]
 
 Since \(K\) is compact one may cover it with finitely many such intervals.
 
 Applying the Besicovitch covering lemma (e.g. see \cite[Theorem 1.1]{deguzman1975}) there exists a constant \(c\) (which does not depend on \(\lambda\) nor \(f\)) such that a subcover may be found so that no more than \(c\) intervals intersect simultaneously.
 
 Summing over such a subcover one has
 \[t\lambda(\lbrace M_\lambda f > t\rbrace) \le 2t\lambda(K) \le 2c\int |f| d\lambda.\]
 
 This inequality has been established for all \(\lambda\)-integrable \(f\) and all \(t > 0\).  Applying it to \(g = f\one{\lbrace |f| > t/2\rbrace}\) one obtains (observing that \(M_\lambda f \le t/2 + M_\lambda g\)) that
 \[t\lambda (\lbrace M_\lambda f > t\rbrace) \le t\lambda(\lbrace M_\lambda g > t/2\rbrace) \le 4c \int |f|\one{\lbrace |f| > t/2\rbrace}d\lambda\]
 which establishes the claim.
\end{proof}

We now use Lemma \ref{maximalinequality} to control the typical maximal function of \(f_k\).  The argument is adapted from \cite[Proposition IV-2-10]{neveu1975}, see the appendix of said work for discussion of this type of results in the context of general Orlicz spaces.
\begin{lemma}[Average maximal function]\label{averagemaximalfunction}
In the context above one has
\[\E{\log\left(M_{\nu_k}f_k(x_k)\right)} < +\infty.\]
\end{lemma}
\begin{proof}
As observed at the beginning of section \ref{sectionorlicz} the conditional distribution of \(x_k\) given \(\sigma(\F_{k-1},T_{k-1},\F_k)\) has density \(f_k\) with respect to \(\nu_k\).  Therefore,
\begin{align*}
\E{\log\left(M_{\nu_k}f_k(x_k)\right)} &= \E{\E{\log\left(M_{\nu_k}f_k(x_k)\right)\vert \F_{k-1},T_{k-1},\F_k}}
\\ &= \E{\int f_k(x)\log\left(M_{\nu_k}f_k(x)\right)d\nu_k(x)}.
\end{align*}

 The lower bound \(f_k\log(f_k) \le f_k\log(M_{\nu_k}f_k)\) which holds \(\nu_k\)-almost everywhere reduces the problem to showing that the expected value on the right is not \(+\infty\).
 
 Applying the inequality \(a\log(b) \le a\log(a) + b/e\) (valid for \(a,b \ge 0\)) one obtains
 \[\E{\int f_k(x)\log\left(M_{\nu_k}f_k(x)\right)d\nu_k(x)} \le \kappa + \frac{1}{e}\E{\int M_{\nu_k}f_k(x)d\nu_k(x)}.\] 
 
 We now conclude by using Lemma \ref{maximalinequality} as follows
 \begin{align*}
  \E{\int M_{\nu_k}f_k(x)d\nu_k(x)} &\le 1 + \E{\int_1^{+\infty}\nu_k\left(\lbrace M_{\nu_k}f_k \ge t\rbrace\right)dt}
  \\ &\le 1 + C\E{\int_1^{+\infty}\int f_k(x)\frac{1}{t}\one{\lbrace f_k \ge t/2\rbrace}d\nu_k(x) dt}
  \\ &\le 1 + C\E{\int f_k(x)\log(2f_k(x))d\nu_k(x)}
  \\ &= 1 + C\log(2) + C\kappa.
 \end{align*}
\end{proof}

\subsubsection{Domination of approximating terms}

We will now establish the main estimate needed to apply Maker's theorem as in Lemma \ref{lengthlemma}.   For the needed upper bound Lemma \ref{averagemaximalfunction} suffices.  For the lower bound we mimic the argument of \cite{chung1961}.

\begin{lemma}\label{dominationlemma}
For each \(n = 1,2,\ldots\) let \(J_{n} = T_{-1}\circ \cdots \circ T_{-n}(I_{-n})\) and
\[X_{n} = \log\left(\frac{T_{-1}\nu_{-1}(J_{n})}{\nu_0(J_n)}\right).\]

Then \(\E{\sup\limits_{n} |X_n|} < +\infty\).
\end{lemma}
\begin{proof}
Notice first that
\[X_n = \log\left(\frac{1}{\nu_0(J_n)}\int\limits_{J_n} f_0(x)d\nu_0(x)\right) \le \log\left(M_{\nu_0}f_0(x_0)\right).\]

In view of Lemma \ref{averagemaximalfunction} this bounds \(\sup\limits_{n}X_n\) from above by an integrable random varaible.

For the lower bound consider the event that \(X_n \le -t\) and notice that this implies
\[\int\limits_{J_n}f_0(x)d\nu_0(x) \le e^{-t}\nu_0(J_n).\]

Given \(f_0\) and \(\nu_0\) define the bad set \(B_t\) as the set of points in the circle belonging to an interval \(I\) such that
\begin{equation}\label{badsetequation}
\int\limits_{I} f_0 d\nu_0 \le e^{-t}\nu_0(I). 
\end{equation}

Following the proof of Lemma \ref{maximalinequality} consider a compact set \(K \subset B_t\) with
\[\int\limits_{B_t} f_0d\nu_0 \le 2\int\limits_{K}f_0d\nu_0.\]

By considering a finite covering of \(K\) by intervals satisfying equation \ref{badsetequation} and summing over a Besicovitch subcover where no more than \(c\) intervals overlap (here the constant \(c\) does not depend on \(f_0\) nor \(\nu_0\)) we obtain:
\[\int\limits_{B_t}f_0\nu_0 \le 2\int\limits_{K}f_0d\nu_0 \le 2ce^{-t}.\]

Using that the conditional distribution of \(x_0\) given \(f_0\) and \(\nu_0\) is \(f_0\nu_0\) we obtain
\begin{align*}
\Pof{\inf\limits_n X_n \le -t} &\le \Pof{x_0 \in B_t} = \E{\Pof{x_0 \in B_t|f_0,\nu_0}} 
\\ &= \E{\int\limits_{B_t}f_0\nu_0} \le 2ce^{-t} 
\end{align*}
which shows that \(\inf\limits_{n}X_n\) is integrable as claimed.
\end{proof}

\subsubsection{Probability estimates}

Having solved the main technical issues we now repeat the argument of Lemma \ref{lengthlemma} replacing the uniform measure \(\eta\) with the random measure \(\nu_0\) to obtains the desired estimate on the \(\nu_0\)-measure of a sequence of intervals shrinking to \(x_0\).

\begin{lemma}[Probability of distinguished intervals]\label{probabilitylemma}
 Almost surely one has
 \[\lim\limits_{n \to +\infty}\frac{1}{n}\log\left(\frac{\nu_{-n}(I_{-n})}{\nu_0(T_{-1}\circ\cdots \circ T_{-n}(I_{-n}))}\right) = \kappa.\]
\end{lemma}
\begin{proof}
 For each \(k \in \Z\) and \(n = 1,2,\ldots\) let \(J_{k,n} = T_{k-1}\circ\cdots \circ T_{k-n}(I_{k-n})\),
\[X_{k,n} = \log\left(\frac{T_{k-1}\nu_{k-1}(J_{k,n})}{\nu_k(J_{k,n})}\right),\]
and
\[X_{k} = \log\left(\frac{dT_{k-1}\nu_{k-1}}{d\nu_k}(x_k)\right).\]

Notice that for each \(n\) the sequence \(X_{k,n}\) is stationary and almost surely \(\lim\limits_{n \to +\infty}X_{k,n} = X_k\).

Furthermore \(\sup\limits_{n} |X_{k,n}|\) is integrable by Lemma \ref{dominationlemma}.

Applying Maker's theorem as in lemma \ref{lengthlemma}, almost surely one has
\begin{align*}
\kappa &= \lim\limits_{n \to +\infty}\frac{1}{n}\sum\limits_{k = 0}^{n-1}\log\left(\frac{dT_{-k-1}\nu_{-k-1}}{d\nu_{-k}}(x_{-k})\right) 
\\ &= \lim\limits_{n \to +\infty}\frac{1}{n}\sum\limits_{k = 0}^{n-1}X_{-k}
\\ &= \lim\limits_{n \to +\infty}\frac{1}{n}\sum\limits_{k = 0}^{n-1}X_{-k,n-k}
\\ &= \lim\limits_{n \to +\infty}\frac{1}{n}\log\left(\frac{\nu_{-n}(I_{-n})}{\nu_0(T_{-1}\circ\cdots \circ T_{-n}(I_{-n}))}\right),
\end{align*}
as claimed.
\end{proof}

A technical issue in what follows is that the asymptotic lower bound for \(\nu_0(T_{-1}\circ \cdots \circ T_{-n}(I_{-n}))\) just obtained, is bad when \(\nu_{-n}(I_{-n})\) is small.  However, in view of Lemma \ref{stationaryintervalslemma}, \(\nu_{-n}(I_{-n}) \ge 1/2\) `half of the time', and this suffices for our needs.

\subsection{Proof of Theorem \ref{dimensiontheorem}}

Let \(n_1 < n_2 < \cdots\) be the (random) sequence of values of  \(n\) for which \(\nu_{-n}(I_{-n}) \ge 1/2\).   By Lemma \ref{stationaryintervalslemma} this occurs with probability at least \(1/2\) for each fixed \(n\).  Hence, by the ergodic theorem, taking a subsequence we may assume that \(n_k = 2k + o(k)\) almost surely.

For each \(k\) let \(J_k = (T_{-1}\circ \cdots \circ T_{-n_k})(I_{-n_k})\).

Fix \(\epsilon > 0\) and let \(r_n = \exp(-(\chi+\epsilon)n)\) and \(R_n = \exp(-(\chi-\epsilon)n)\).

Choose two integer valued functions \(\ell(r) \le k(r)\) such that 
\[\ell(r) = \frac{-(1-\epsilon)\log(r)}{2(\chi+\epsilon)} + o(\log(r))\]
and
\[k(r) = \frac{-(1+\epsilon)\log(r)}{2(\chi-\epsilon)} + o(\log(r))\]
as \(r \to 0\).

Notice that eventually one has \(R_{n_{k(r)}} \le r \le r_{n_{\ell(r)}}\) and therefore by Lemma \ref{lengthlemma} almost surely
\[J_{k(r)} \subset B_{R_{n_{k(r)}}}(x_0) \subset B_r \subset B_{r_{n_{\ell(r)}}} \subset J_{\ell(r)},\]
for all \(r\) small enough.

Combining these facts one obtains the bounds 
\[\frac{-\log(\nu_0(J_{\ell(r)}))}{-\log(R_{n_{k(r)}})} \le \frac{-\log(\nu_0(B_r(x_0)))}{-\log(r)} \le \frac{-\log(\nu_0(J_{k(r)}))}{-\log(r_{n_{\ell(r)}})}\]

By Lemma \ref{probabilitylemma} almost surely
\[-\log(\nu_0(J_{k})) = \kappa n_k + o(k),\]
when \(k \to +\infty\).

This implies that almost surely 
\[\frac{(1-\epsilon)\kappa}{\chi + \epsilon} \le  \lowerdim_{x_0}(\nu) \le \upperdim_{x_0}(\nu) \le \frac{(1+\epsilon)\kappa}{\chi -\epsilon}.\]

By intersecting over the corresponding full measure sets for a countable sequence \(\epsilon_n \to 0\) one obtains that almost surely \(\nu_0\) is exact dimensional with dimension \(\kappa/\chi\) as claimed.

\end{document}